\documentclass[11pt,a4paper]{article}

\usepackage{latexsym}
\usepackage{blindtext}

\PassOptionsToPackage{force}{filehook}

\usepackage[utf8]{inputenc}
\usepackage[english]{babel}
\usepackage{arabtex}
\usepackage{utf8}
\usepackage{graphicx}
\usepackage{graphics,subfigure}
\usepackage{fancyhdr}
\usepackage{lmodern}
\usepackage{color}
\usepackage{transparent}
\usepackage{float}
\usepackage{pythontex}
\usepackage{tikz-cd}
\usepackage{listings}
\usepackage[linktocpage=true,citebordercolor=green, urlcolor=black]{hyperref}
\usepackage{tabularx}
\usepackage[title,titletoc]{appendix}
\usepackage{xcolor}
\usepackage[mathscr]{eucal}
\usepackage{calrsfs}
\usepackage{cite}
\usepackage{comment}

\usetikzlibrary{shapes.misc}
\usepackage{pgf,tikz}
\usepackage{caption}
\usetikzlibrary{decorations.markings}
\usetikzlibrary{arrows}
\usetikzlibrary{patterns}
\usepackage{tikz-3dplot}
\usetikzlibrary{calc}
\usepackage{amsfonts}
\usepackage{amsmath,amssymb,amsthm}
\usepackage{mathtools}
\usepackage{relsize}
\usepackage{titlesec}

\usepackage{stmaryrd}
\usepackage{trimclip}

\usepackage{enumitem}

\usepackage{wasysym}

\makeatletter
\DeclareRobustCommand{\shortto}{%
	\mathrel{\mathpalette\short@to\relax}%
}

\newcommand{\short@to}[2]{%
	\mkern2mu
	\clipbox{{.5\width} 0 0 0}{$\m@th#1\vphantom{+}{\shortrightarrow}$}%
}
\makeatother

\theoremstyle{definition}
\newtheorem{definition}{Definition}[section]
\newtheorem{example}[definition]{Example}
\newtheorem{remark}[definition]{Remark}
\newtheorem{corollary1}[definition]{Corollary}
\newtheorem{lemma1}[definition]{Lemma}
\newtheorem{theorem1}[definition]{Theorem}
\newtheorem{Proposition}[definition]{Proposition}
\newtheorem*{proposition*}{Proposition}

\titleformat{\paragraph}
{\normalfont\normalsize\bfseries}{\theparagraph}{1em}{}
\titlespacing*{\paragraph}
{0pt}{3.25ex plus 1ex minus .2ex}{1.5ex plus .2ex}

\usepackage[margin=2.5cm]{geometry}

\usepackage{mathptmx}

\AtEndDocument{%
	\par
	\medskip
	\begin{tabular}{@{}l@{}}%
		\textsc{Cengiz Aydin}\\
		\textsc{Institut für Mathematik, Universität Heidelberg, Germany}\\
		\textit{E-mail address}: \texttt{\href{mailto:cengiz.aydin@hotmail.de}{cengiz.aydin@hotmail.de}}
\end{tabular}}

\usepackage{authblk}
\title{Contact geometry of Hill's approximation in a spatial restricted four-body problem}
\author{Cengiz Aydin}

\begin{document}
	
\setcounter{page}{1}
\pagenumbering{arabic}
	
\maketitle
	
\begin{abstract}
It is well-known that the planar and spatial circular restricted three-body problem (CR3BP) is of contact type for all energy values below the first critical value.\ Burgos-García and Gidea extended Hill's approach in the CR3BP to the spatial equilateral CR4BP, which can be used to approximate the dynamics of a small body near a Trojan asteroid of a Sun--planet system.\ Our main result in this paper is that this Hill four-body system also has the contact property.\ In other words, we can ``contact'' the Trojan.\ Such a result enables to use holomorphic curve techniques and Floer theoretical tools in this dynamical system in the energy range where the contact property holds.\
\end{abstract}

\begin{center}
\begin{tabular}{ll}
	\textbf{Keywords} & $\quad$ four-body problem $\cdot$ Hill's approximation $\cdot$ celestial mechanics $\cdot$ contact geometry \\
	\textbf{MSC 2020} & $\quad$ 70G45 $\cdot$ 70F10 $\cdot$ 53D35
\end{tabular}
\end{center}
		
\tableofcontents
		
\section{Introduction}


\textbf{Astronomical significance.}\ One of the first triumphs in celestial mechanics was the Lagrange central configuration, one of the first explicit solutions to the three-body problem discovered by Lagrange \cite{lagrange} in 1772.\ It consists of three bodies, not necessarily of equal masses, forming the vertices of an equilateral triangle, each moving on a specific Kepler orbit.\ The triangular configuration of the bodies is maintained throughout the entire motion.\ A~special type of Lagrange's solution is the rigid circular motion of the three bodies around their center of mass.\ It is common to use the term ``Trojan'' to describe a small body, an asteroid or a natural satellite, that lies in such equilateral triangular configuration together with the Sun and a planet, or with a planet and a moon.\ In other words, such small bodies remain near triangular points 60° ahead of or behind the orbit of a planet or a moon.\ Such triangular points correspond to the two equilateral Lagrange points, $L_4$ (leading) and $L_5$ (trailing), of a Sun--planet or a planet--moon system.\ Since the discovery of the first Trojan asteroid, 588 Achilles, near Jupiter's Lagrange point $L_4$ by Max Wolf of the Heidelberg Observatory in 1906 (see \cite{nicholson}), such configurations have not only deserved attention in theory, but have also gained tremendous astronomical significance.\ By now many other examples of Trojan-like asteroids in our solar system have become known.\ Jupiter has thousands of Trojans \cite{stacey_connors}; Mars \cite{connors} and Neptune \cite{almeida} also have some; only two Earth Trojans have been discovered so far \cite{santana}.\ Meanwhile it is known \cite{murray} that the Saturn--Tethys system has two Trojans, Telesto ($L_4$-Trojan) and Calypso ($L_5$-Trojan), and the Saturn--Dione system has two as well, Helene ($L_4$-Trojan) and Polydeuces ($L_5$-Trojan).\ A twelve-year space probe to several Jupiter Trojans is currently being operated by NASA's Lucy mission, which was launched on 16 October 2021 as the first mission to the Jupiter Trojans (see e.g.,\ \cite{olkin} for a recent research result).\ Outside the solar system there exists also the possibility of a Trojan planet associated to extrasolar systems, formed by a star with similar mass as the Sun and a giant gas planet.\ Although such Trojan planets only play a fictitious role at the moment, their dynamics are already being analyzed theoretically \cite{schwarz_dvorak}.\

In order to describe conveniently the dynamics of small bodies attracted by the gravitational field of three bodies in such a triangular central configuration, a restricted four-body problem (R4BP) becomes necessary.\ There are plenty of results on various models of the R4BP, such as \cite{alvarez_vidal}, \cite{baltagiannis_2011}, \cite{baltagiannis_2013}, \cite{burgos_delgado,burgos_celletti,burgos_gidea}, \cite{cronin}, \cite{gabern_jorba}, \cite{kumari}, \cite{michalodimitrakis}, \cite{moulton}, \cite{scheeres}, \cite{steves_roy}.\ Relevant for this work is the spatial equilateral circular one, in which three primaries circle around their common center of mass and forming an equilateral triangular configuration.\ In view of astronomical data associated to such configurations in the solar system, the mass of one of the primaries (the Trojan) is much smaller than the other two primaries.\ If one equates the mass of the Trojan to zero, the system represents the circular restricted three-body problem (CR3BP).\ Therefore, to study the dynamics in the vicinity of the Trojan, a practical and intelligent concept is to perform a Hill's approximation in the equilateral circular R4BP.\vspace{0.5em}\\ 
\textbf{Hill's approximation.}\ In 1878 Hill \cite{hill} introduced a limiting case of the CR3BP, as an approach to solve the motion of the Moon in the Sun--Earth problem.\ As a first approximation, the infinitesimal body (Moon) moves in the vicinity of the smaller primary (Earth) and, by a symplectic rescaling of coordinates, the remaining primary (Sun) is pushed infinitely far away in a way that it acts as a velocity independent gravitational perturbation of the rotating Kepler problem formed by the Earth and the Moon.\

Extending Hill's concept to the equilateral circular R4BP was performed by Burgos-García and Gidea \cite{burgos_gidea}, which is the central system in this paper.\ This problem studies the dynamics near the Trojan and pushes the two remaining primaries (e.g.,\ Sun and Jupiter) to infinity, and depends on two parameters, the energy of the system and the mass ratio $\mu \in [0,\frac{1}{2}]$ of the two primaries at infinity (system is symmetric with respect to $\mu = \frac{1}{2}$).\ The case $\mu = 0$ corresponds to the classical Hill 3BP, therefore this Hill four-body model generalizes the classical Hill's approach.\ It is worth noting that this system is different as the one introduced by Scheeres \cite{scheeres}, in which the motion of a spacecraft in the Sun perturbed Earth--Moon system is studied.\ Moreover, this Hill four-body system was extended in \cite{burgos_celletti} as a problem with oblate bodies modeling the Sun--Jupiter--Hektor--Skamandrios system.\vspace{0.5em}\\
\textbf{Why we care about contact property.}\ One of Hill's main contributions was the discovery of one periodic solution with a period of one synodic month of the Moon.\ Hill's lunar theory was, as Wintner said \cite[p.\ 1]{wintner}, \textit{``considered by Poincaré as representing a turning point in the history of celestial mechanics''}.\ Poincaré sought to make periodic solutions central in the study of the global dynamics, a focus that has persisted since his pioneering work \cite{poincare}.\ Inspired by Poincaré's concept of using global surface of sections for proving existence results of periodic orbits in the CR3BP \cite{poincare_2}, Birkhoff conjectured \cite{birkhoff} that retrograde periodic orbits in the CR3BP bound a disk-like global surface of section (retrograde means that the motion of the orbit is in opposite direction as the coordinate system is rotating; direct is the one that rotates in the same direction).\ Due to preservation of an area form with finite total area, one can apply Brouwer's translation theorem to the Poincaré return map associated to the disk-like global surface of section and find at least one fixed point that should correspond to a direct periodic orbit.\ The direct orbit is astronomically more significant, since our Moon moves in a direct motion around the Earth, whose existence is based so far on numerical computations, as Hill's lunar orbit.\ Such fixed point approaches related to existence results of periodic orbits are sources of inspiration that have laid the fruitful fundamental principles of powerful abstract methods and important breakthroughs in symplectic and contact geometry, such as the work by Floer \cite{floer} on the Arnold conjecture, by Hofer \cite{hofer} and Taubes \cite{taubes} on the Weinstein conjecture, and by Hofer--Wysocki--Zehnder \cite{hofer_w_z} on the construction of disk-like global surface of sections with the help of holomorphic curves.\ The assumption that energy level sets are of contact type enable to use holomorphic curve and Floer theoretical techniques in the energy range where the contact property holds.\ Especially, many recent deep results associated with the dynamics of low-dimensional contact manifolds have been proved:\ In the 3-dimensional case, \cite{colin_dehornoy_rechtman} proved the existence of supporting broken book decompositions, \cite{colin_dehornoy_hryniewicz_rechtman} showed how to use these broken book decompositions to construct Birkhoff sections or global surfaces of section, \cite{cristofaro_hryn_hutchings_liu} gave a detailed description of the dynamics when there are exactly two simple Reeb orbits, and \cite{irie} described an abstract framework for proving strong closing properties based on the smooth closing lemma for Reeb flows; in the 5-dimensional case, \cite{moreno_koert} showed the relation between the spatial dynamics of the CR3BP and iterated open book decompositions.\ We also refer to \cite{frauenfelder_koert} for a profound introduction to holomorphic techniques and their use in celestial mechanics, particularly focused on the CR3BP and the above Birkhoff's conjecture.\ Another dynamical consequence of the contact property of energy level sets, discussed in the latter reference, is that blue sky catastrophes cannot occur.\ On a practical level, Floer theoretic bifurcation tools have recently been applied to numerical investigations of periodic orbits \cite{aydin_babylonian}, \cite{aydin_cz}, \cite{moreno_aydin_koert_frauenfelder_koh}.\vspace{0.5em}\\
\textbf{Main result.}\ For the planar CR3BP it is well-known that below the first critical value, the two bounded components of the energy level sets, after Moser regularization, are of contact type \cite{albers_frauenfelder_koert_paternain}.\ Each component corresponds to the unit cotangent bundle of $S^2$ with the standard contact structure, meaning that each contact manifold corresponds to $(S^*S^2,\xi_{st})$.\ The same result for the spatial case was shown in \cite{cho_jung_kim}, where each contact manifold corresponds to $(S^*S^3,\xi_{st})$.\ We note that \cite[Chapter 6.1]{frauenfelder_koert} proved the same result for the classical planar Hill 3BP.\

The Hill four-body system we consider has four Lagrange points, where $L_1$ is symmetric to $L_2$ (lying on the $x$-axis), and $L_3$ is symmetric to $L_4$ (lying on the $y$-axis).\ If the energy value $c$ is below the first critical value $H(L_{1/2})$, then the energy level set has one bounded component (where the origin is contained), which we denote by $\Sigma_c^b$.\ This component is non-compact because of a singularity at the origin corresponding to collision.\ After performing Moser regularization, we obtain a compact 5-dimensional manifold, which we denote by $\widetilde{\Sigma}_c^b$.\ The spatial system is invariant under a symplectic involution $\sigma$ which is induced by the reflection at the ecliptic.\ The restriction of the spatial problem to the fixed point set Fix($\sigma$) corresponds to the planar problem.\ In fact, we can restrict the whole procedure to Fix($\sigma$) and obtain a compact 3-dimensional manifold,
which we denote by $\widetilde{\Sigma}_c^b|_{\text{Fix}(\sigma)}$.\ Our main result in this paper is the following theorem.\
\begin{theorem1}\label{theorem}
	\textit{For any given $\mu \in [0,\frac{1}{2}]$ it holds that}
	\begin{align*}
		\widetilde{\Sigma}_c^b &\cong (S^*S^3,\xi_{st}),\quad \textit{if } c < H(L_{1/2}), \\
		\widetilde{\Sigma}_c^b|_{\text{Fix}(\sigma)} &\cong (S^*S^2,\xi_{st}),\quad \textit{if } c < H(L_{1/2}).
	\end{align*}
\end{theorem1}
Our method to prove Theorem \ref{theorem} is the same as in \cite{albers_frauenfelder_koert_paternain}, \cite{cho_jung_kim}, namely we find a Liouville vector field on the cotangent bundle which is transverse to $\widetilde{\Sigma}_c^b$ whenever $c < H(L_{1/2})$.\ This transversality result implies the contact property.\ The Liouville vector field we use is inspired by Moser regularization, which first interchanges the roles of position and momenta, and then uses the stereographic projection.\ In this setting, the Liouville vector field is the natural one (i.e.,\ the radial vector field in fiber direction) on the new cotangent bundle structure after switching position and momenta.\ Therefore, our transversality result implies in particular fiberwise starshapedness.\vspace{0.5em}\\
\textbf{Organization of the paper.}\ In Section \ref{sec2} we discuss the Hamiltonian, its linear symmetries, Lagrange points and Hill's regions.\ The goal of Section \ref{sec3} is to prove Theorem \ref{theorem}.\ We first recall some basic definitions and notations from contact geometry, and then show transversality in the non-regularized case.\ After this, we perform Moser regularization and prove therein the transversality property.\

\section{Hill's approximation in the spatial equilateral circular R4BP}
\label{sec2}

\subsection{Hamiltonian}

We consider three point masses (primaries), $B_1$, $B_2$ and $B_3$, moving in circular periodic orbits in the same plane with constant angular velocity around their common center of gravity fixed at the origin, while forming an equilateral triangle configuration (see Figure \ref{figure_1_r4bp}).\ A fourth body $B_4$ is significantly smaller than the other three and thus a negligible effect on their motion.\ We set $B_1$ on the negative $x$-axis at the origin of time and assume that the corresponding three masses are $m_1 \geq m_2 \geq m_3$.\ It is convenient to choose the units of mass, distance and time such that the gravitational constant is~1, and the period of the circular orbits is $2\pi$.\ In these units the side length of the equilateral triangle configuration is normalized to be one, and $m_1 + m_2 + m_3 = 1$.\ Moreover, it is convenient to use a rotating frame of reference that rotates with an angular velocity of the orbital angular rate of the primaries.\ Then, the dynamics of the infinitesimal body $B_4$ is described by the Hamiltonian
$$ H(x,y,z,p_x,p_y,p_z) = \frac{1}{2} \left( p_x^2 + p_y^2 + p_z^2 \right) - \frac{m_1}{r_1} - \frac{m_2}{r_2} - \frac{m_3}{r_3} + p_x y - p_y x, $$
which is a first integral of the system.\ An equivalent first integral is the Jacobi integral $C$ defined by $C = -2H$.\ Notice that $r_i$ indicates the corresponding distance from $B_4$ to $i$-th primary, for $i=1,2,3$.\ The general expressions of the position coordinates $(x_i,y_i,0)$ can be seen in \cite{baltagiannis_2013}.\ If $m_3=0$ and $m_2 = \mu$, then one recovers the constellation of the CR3BP associated to $B_1$ and $B_2$, where $B_3$ is located at the equilateral Lagrange point $L_4$.\ Moreover, the phase space is the trivial cotangent bundle $T^* \left( \mathbb{R}^3 \setminus \{ B_1,B_2,B_3 \} \right) =  \left( \mathbb{R}^3 \setminus \{ B_1,B_2,B_3 \} \right) \times \mathbb{R}^3$, endowed with the standard symplectic form $\omega = \sum d p_k \wedge d k$ ($k=x,y,z$).\ The flow of the Hamiltonian vector field $X_H$, defined by $dH(\cdot) = \omega( \cdot , X_H )$, is equivalent to the equations of motion, $\left\{ \dot{k} = \frac{\partial H}{\partial p_k}, \dot{p}_k = - \frac{\partial H}{\partial k} \right\}$ ($k=x,y,z$).\
\begin{figure}[t]
	\centering
	\includegraphics[scale=1]{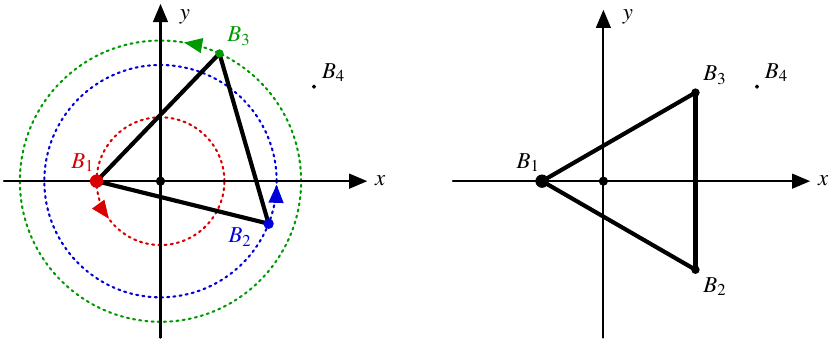}
	\caption{Equilateral circular restricted four-body problem.\ Left:\ Case of $m_1 > m_2 > m_3$.\ Right:\ Case of $m_2 = m_3$ in a rotating frame of reference; $B_2$ and $B_3$ are located symmetrically with respect to $B_1$.}
	\label{figure_1_r4bp}
\end{figure}

We now briefly recall the fundamental steps of Hill's approximation, as performed in \cite{burgos_gidea} where the details can be seen.\ Let $B_3$ be the primary (the Trojan), whose mass is much smaller than the other two primaries.\ The first step is to set the Trojan to the origin.\ The second step rescales symplectically the coordinates depending on $m_3^{1/3}$.\ The third step makes use of a Taylor expansion of the gravitational potential of the Hamiltonian in powers of $m_3^{1/3}$.\ Finally, the limiting case for $m_3 \to 0$ yields the Hamiltonian
$$ H (x,y,z,p_x,p_y,p_z) = \frac{1}{2} \left( p_x^2 + p_y^2 + p_z^2 \right) + p_x y - p_y x - \frac{1}{r} + \frac{1}{8}x^2 - \frac{3\sqrt{3}}{4}(1-2\mu)xy -\frac{5}{8}y^2 + \frac{1}{2}z^2, $$
where $r = \left( x^2 + y^2 + z^2 \right)^{\frac{1}{2}}$, $m_1 = 1 - \mu$ and $m_2 = \mu$.\ Notice that if one expands the Hamiltonian of the CR3BP centered at the equilateral Lagrange point $L_4$, then the quadratic part corresponds to $H + 1/r$.\\
Furthermore, after applying a rotation in the $xy$-plane, the system is equivalent with the system characterized by the Hamiltonian
\begin{align}\label{hamiltonian_1}
	H (x,y,z,p_x,p_y,p_z) = \frac{1}{2} \left( p_x^2 + p_y^2 + p_z^2 \right) + p_x y - p_y x - \frac{1}{r} + a x^2 + b y^2 + \frac{1}{2}z^2,
\end{align}
\begin{figure}[t]
	\centering
	\includegraphics[scale=1]{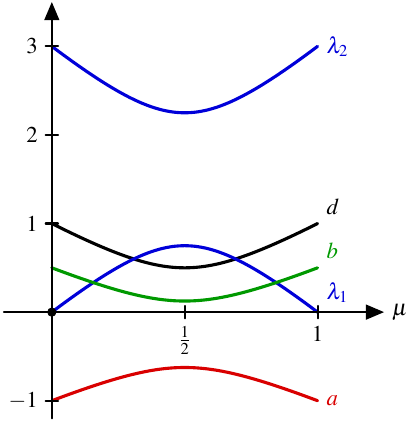}
	\caption{The quantities $a$ (red), $b$ (green), $\lambda_1$, $\lambda_2$ (both blue) and $d$ (black).}
	\label{figure_2_quantities}
\end{figure}
where
$$ a = \frac{1}{2}(1 - \lambda_2),\quad b = \frac{1}{2} (1-\lambda_1),\quad \lambda_1 = \frac{3}{2}(1-d),\quad \lambda_2 = \frac{3}{2}(1+d),\quad d = \sqrt{1 - 3\mu + 3\mu^2}. $$
Since $d(1-\mu)=d(\mu)$, we can assume that $\mu \in [0,\frac{1}{2}]$.\ Notice that $\lambda_1$ and $\lambda_2$ are the eigenvalues corresponding to the rotation transformation in the $xy$-plane.\ The quantities $a,b,\lambda_1,\lambda_2$ and $d$ are plotted in Figure \ref{figure_2_quantities}.\ The Hamiltonian (\ref{hamiltonian_1}) consists of the rotating Kepler problem (formed by the Trojan and the infinitesimal body) with a velocity independent gravitational perturbation produced by the two remaining massive primaries (the degree 2 term $a x^2 + b y^2 + \frac{1}{2}z^2$) which are sent at infinite distance.\ By introducing the effective potential
\begin{align}\label{effective_potential}
	U \colon \mathbb{R}^3 \setminus \{ 0 \} \to \mathbb{R},\quad (x,y,z) \mapsto - \frac{1}{r} - \frac{1}{2} \left( \lambda_2 x^2 + \lambda_1 y^2 - z^2 \right),
\end{align}
the Hamiltonian (\ref{hamiltonian_1}) can be written as
\begin{align}\label{hamiltonian_2}
	H (x,y,z,p_x,p_y,p_z) = \frac{1}{2} \left( (p_x + y)^2 + (p_y - x)^2 + p_z^2 \right) + U(x,y,z),
\end{align}
and the equations of motion are given by
\begin{align}\label{hamiltonian_equation}
	\ddot{x} - 2 \dot{y} &= - \frac{\partial U}{\partial x} =  \left(\lambda_2 - \frac{1}{r^3} \right)x \nonumber\\
	\ddot{y} + 2 \dot{x} &= - \frac{\partial U}{\partial y} = \left(\lambda_1 - \frac{1}{r^3} \right)y\\
	\ddot{z} &= - \frac{\partial U}{\partial z} = - \left(1 + \frac{1}{r^3} \right)z. \nonumber
\end{align}
In particular, the case $\mu = 0$ recovers the classical Hill 3BP.\ While the Hill 3BP depends only on the energy of the orbit, this systems depends on two parameters, the mass ratio $\mu$ and the energy of the system.\ Specific $\mu$-values of practical interest are for example $\mu = 0.00095$, which approximates the Sun--Jupiter mass ratio, and $\mu = 0.00547$, which corresponds to the extrasolar system associated to the Sun-like star HD 28185 and its Jupiter-like exoplanet HD 28185 b.\

\subsection{Linear symmetries}

A ``symmetry'' $\sigma$ is, by definition, a symplectic or anti-symplectic involution of the phase space which leaves the Hamiltonian invariant, i.e.,\
\begin{align}\label{symmetry_equations}
	H \circ \sigma = H,\quad \sigma^2 = \text{id}, \quad \sigma^* \omega = \pm \omega.
\end{align}
Anti-symplectic symmetries denote time-reversal symmetries in the Hamiltonian context, see e.g., \cite{lamb_roberts}.\ A periodic solution $\mathbf x \equiv (x,y,z,p_x,p_y,p_z)$ is symmetric with respect to an anti-symplectic symmetry $\rho$ if $\mathbf x (t) = \rho \left( \mathbf x (-t) \right)$ for all $t$, and symmetric with respect to a symplectic one $\sigma$ if $\mathbf x (t) = \sigma \left( \mathbf x (t) \right)$ for all $t$.\

The reflection at the ecliptic $\{z = 0\}$ gives rise to a linear symplectic symmetry of (\ref{hamiltonian_1}), denoted by
\begin{align}\label{sigma}
	\sigma (x,y,z,p_x,p_y,p_z) = (x,y,-z,p_x,p_y,-p_z),
\end{align}
whose fixed point set Fix$(\sigma) = \{ (x,y,0,p_x,p_y,0) \}$ corresponds to the planar problem.\ Other linear symplectic symmetries are $- \sigma$ and $\pm \text{id}$, where $-\sigma$ corresponds to the $\pi$-rotation around the $z$-axis, hence the $z$-axis is invariant under $-\sigma$.\ Linear anti-symplectic symmetries are determined by
\begin{itemize}[noitemsep]
	\item $\rho_1 (x,y,z,p_x,p_y,p_z) = (x,-y,-z,-p_x,p_y,p_z)$ ($\pi$-rotation around the $x$-axis),
	\item $\rho_2 (x,y,z,p_x,p_y,p_z) = (x,-y,z,-p_x,p_y,-p_z)$ (reflection at the $xz$-plane),
	\item $\rho_3 (x,y,z,p_x,p_y,p_z) = (-x,y,-z,p_x,-p_y,p_z)$ ($\pi$-rotation around the $y$-axis),
	\item $\rho_4 (x,y,z,p_x,p_y,p_z) = (-x,y,z,p_x,-p_y,-p_z)$ (reflection at the $yz$-plane).
\end{itemize}
Together with the previous linear symplectic symmetries, they form the group $\mathbb{Z}_2 \times \mathbb{Z}_2 \times \mathbb{Z}_2$.\ If one restrict the system to Fix($\sigma$), linear anti-symplectic symmetries for the planar problem are given by
\begin{itemize}[noitemsep]
	\item $\rho_x (x,y,0,p_x,p_y,0) = (x,-y,0,-p_x,p_y,0)$ (reflection at the $x$-axis),
	\item $\rho_y (x,y,0,p_x,p_y,0) = (-x,y,0,p_x,-p_y,0)$ (reflection at the $y$-axis),
\end{itemize}
that together with the linear symplectic ones $\{ \pm \text{id} \}$ form a Klein-four group $\mathbb{Z}_2 \times \mathbb{Z}_2$.\ These symmetries show that it is not possible to say which of the two primaries at infinity we are moving to or away from.\

\begin{remark}
	In \cite{aydin_sym} it shown that the Hill 3BP ($\mu = 0$) has two special properties.\
	\begin{itemize}[noitemsep]
		\item[i)] \textit{The spatial linear symmetries already determine the planar ones.}\ The same phenomenon is also true for all $\mu \in [0,\frac{1}{2}]$.\ To see this, let us denote by $\Sigma_s$ and $\Sigma_p$ each set of spatial and planar linear symmetries.\ Consider the projection map given by the restriction to $\text{Fix}(\sigma)$,
		$$ \pi \colon \Sigma_s \to \Sigma_p,\quad \rho \mapsto \rho | _{\text{Fix}(\sigma)}. $$
		If $\rho \in \Sigma_s$, then $\rho | _{\text{Fix}(\sigma)} \in \Sigma_p$ with the corresponding (anti-)symplectic property.\ While $\pi$ is not injective (since $\pi (\rho_1) = \pi (\rho_2)$), it is surjective.\ If $\rho \in \Sigma_p$ is symplectic (or anti-symplectic), then a symplectic (or anti-symplectic) extension is given by $z \mapsto z $ and $p_z \mapsto p_z$ (or $z\mapsto -z$ and $p_z \mapsto p_z$).\
		\item[ii)] \textit{There are no other linear symmetries.}\ This statement also holds for all $\mu \in [0,\frac{1}{2}]$.\ Its proof uses the equations (\ref{symmetry_equations}) and the properties of linear symplectic and anti-symplectic involutions.\ Since the exact same computations work for (\ref{hamiltonian_1}) for all $\mu \in [0,\frac{1}{2}]$, we forgo its proof in this paper.\ 
	\end{itemize}
\end{remark}

\subsection{Lagrange points and Hill's region}

From the third equation in (\ref{hamiltonian_equation}) it is obvious that all Lagrange points are located at the ecliptic $\{z=0\}$.\ Using the projection onto the configuration space given by
\begin{align}\label{projection}
	\pi \colon \mathbb{R}^3 \setminus \{ 0 \} \times \mathbb{R}^3 \to \mathbb{R}^3 \setminus \{ 0 \},\quad (x,y,z,p_x,p_y,p_z) \mapsto (x,y,z),
\end{align}
there is a one-to-one correspondence between critical points of the Hamiltonian (\ref{hamiltonian_2}) and the effective potential (\ref{effective_potential}), determined by $ \left( \pi |_{\text{crit}(H)} \right)^{-1} (x,y,0) = (x,y,0,-y,x,0)$.\ In \cite{burgos_gidea} it is shown that (\ref{effective_potential}) has four critical points, whose coordinates are given explicitly in terms of $\mu$,
$$ L_1 = \left(  \frac{1}{\sqrt[3]{\lambda_2}},0,0 \right),\quad  L_2 = \left( -  \frac{1}{\sqrt[3]{\lambda_2}},0,0 \right),\quad  L_3 = \left( 0, \frac{1}{\sqrt[3]{\lambda_1}},0 \right),\quad L_4 = \left( 0, - \frac{1}{\sqrt[3]{\lambda_1}},0 \right). $$
Note that $L_{1/2}$ are related to each other by $\rho_y$ (reflection at the $y$-axis), and $L_{3/4}$ are related to each other by $\rho_x$ (reflection at the $x$-axis).\ The classical Hill 3BP ($\mu=0$) only has $L_{1/2}$, and especially, if $\mu \to 0$ then $\lambda_1 \to 0$, which means that $L_3$ and $L_4$ are sent to infinity.\ Therefore, the presence of a second primary at infinity for $\mu \in (0,\frac{1}{2}]$ produces the two additional Lagrange points $L_{3/4}$.\ Since $\lambda_2 > \lambda_1$, we have for the critical values
$$ H(L_{1/2}) = - \frac{3}{2} \sqrt[3]{\lambda_2}  < - \frac{3}{2} \sqrt[3]{\lambda_1} =  H(L_{3/4}),\quad \text{for all } \mu \in (0,\frac{1}{2}]. $$
We now consider the energy level set $\Sigma_c := H^{-1}(c)$, for $c \in \mathbb{R}$.\ In view of the footpoint projection (\ref{projection}), the ``Hill's region'' of $\Sigma_c$ is defined as
$$ \mathscr{K}_c := \pi (\Sigma_c) \subset \mathbb{R}^3 \setminus \{ 0 \}, $$
which means that the Hill's region of the energy level set is its shadow under the footpoint projection.\ Since the first three terms in (\ref{hamiltonian_2}) are quadratic and hence non-negative, we can obtain the Hill's region by
$$ \mathscr{K}_c = \left\{ (x,y,z) \in \mathbb{R}^3 \setminus \{0\} \mid U(x,y,z) \leq c \right\}. $$
The topology of the Hill’s region depends on the energy level.\ If $c < H(L_{1/2})$, then the Hill's region has two connected components, one bounded and one unbounded (see Figure \ref{figure_3_hills_region}).\ We denote the bounded component by $\mathscr{K}_c^b$ and abbreviate by
\begin{align}\label{bounded_component}
	\Sigma_c^b := \pi^{-1} (\mathscr{K}_c^b) \cap \Sigma_c
\end{align}
the corresponding connected component of $\Sigma_c$.\
\begin{figure}[t]
	\centering
	\includegraphics[width=1\linewidth]{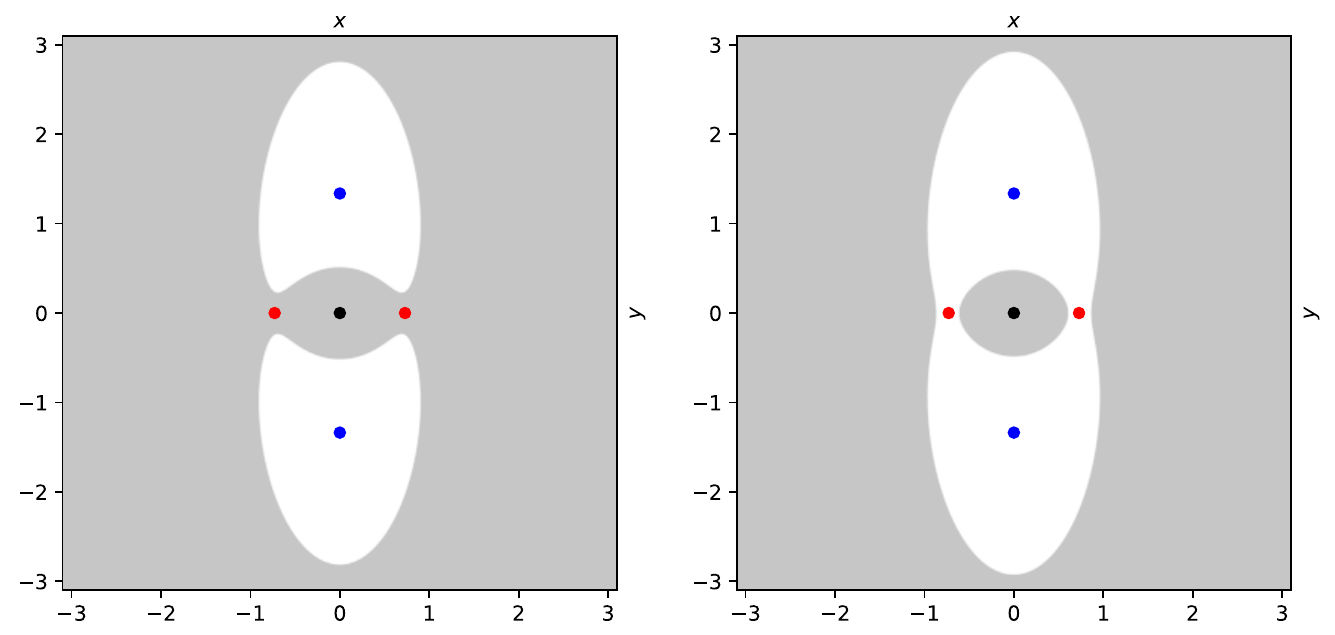}
	\caption{Hill's region (gray shaded domains) for planar problem $\{z=0\}$ for $\mu=0.2$.\ White domains correspond to forbidden regions.\ Red dots indicate $L_{1/2}$; blue dots indicate $L_{3/4}$.\ Right:\ For $c < H(L_{1/2})$.\ Left:\ For $H(L_{1/2}) < c < H(L_{3/4})$.\ In the Hill 3BP ($\mu=0$), when $L_{3/4}$ are sent to infinity, below the critical value the Hill's region consists of one bounded component and two unbounded components.}
	\label{figure_3_hills_region}
\end{figure}

\section{Contact property - Proof of Theorem \ref{theorem}}
\label{sec3}

\subsection{Basic notations}

We now recall some basic definitions and notations from contact geometry, and refer for details to \cite{geiges}.\

\begin{definition}
	Let $M$ be a smooth manifold of odd dimension $2n +1$.\ A ``contact form'' on $M$ is a 1-form $\alpha \in \Omega^1(M)$ such that $\alpha \wedge (d\alpha)^{\wedge n} \neq 0$.\ Given a contact form $\alpha$, the hyperplane field $\xi = \text{ker}\alpha \subset TM$ is oriented by $d\alpha$, and this oriented codimension-1 field is called the ``contact structure''.\ The pair $(M,\xi)$ is called ``contact manifold''.\ The ``Reeb vector field'' $R_{\alpha}$ is the unique vector field defined by the equations $ d \alpha (R_{\alpha}, \cdot) = 0$ and $\alpha(R) = 1$, whose flow is called ``Reeb flow''.\
\end{definition}
\begin{definition}
	A ``Liouville vector field'' $X$ on a symplectic manifold $(M,\omega)$ is a vector field satisfying $ \mathscr{L} _X \omega = \omega $, where $\mathscr{L}$ denotes the Lie derivative, i.e.,\ the Lie derivative along $X$ preserves $\omega$.\
\end{definition}
By Cartan's formula and the closedness of the symplectic form $\omega$, we have $ \mathscr{L} _X \omega = d \left( \iota_X \omega \right) + \iota_X d \omega = d \left( \iota_X \omega \right)$ and therefore, we can write the Liouville condition as $d \left( \iota_X \omega \right) = \omega$, where $\iota_X \omega (\cdot) = \omega(X,\cdot) $.\
\begin{example}
	The cotangent bundle $T^* Q$ of a smooth manifold $Q$ of dimension $n$ is endowed with the so-called ``Liouville one-form''.\ In local coordinates $(q_1,...,q_n)$ on $Q$ and dual coordinates $(p_1,...,p_n)$ on the fibers of $T^* Q$, the Liouville one-form is defined by $\lambda_{can} = \sum _{i=1} ^n p_i dq_i$.\ Since the standard symplectic form is characterized by $\omega_{can} = d \lambda_{can} = \sum_{i=1}^n dp_i \wedge dq_i$, the ``natural Liouville vector field'' $X$ on $T^* Q$ associated to $\lambda_{can}$ is defined by $\iota_X \omega_{can} = \lambda_{can}$.\ In local coordinates,
	$$X = \sum_{i=1}^n p_i \frac{\partial}{\partial p_i},$$
	that is, the radial vector field in fiber direction.\
\end{example}
\noindent
\textbf{Hypersurfaces of contact type.}\ Let $X$ be a Liouville vector field on a $2n+2$ dimensional symplectic manifold $(M,\omega)$.\ Then $\alpha := \iota_X \omega |_{\Sigma} $ is a contact form on any hypersurface $\Sigma \subset M$ transverse to $X$ (i.e.,\ with $X$ nowhere tangent to $\Sigma$).\ Such hypersurfaces are said to be of ``contact type''.\ To see this, let $x \in \Sigma$ and let $\{v_1,...,v_{2n+1} \}$ be a basis of $T_x\Sigma$.\ By using the Liouville condition we have,
\begin{align}\label{contact}
	\alpha \wedge (d \alpha)^{\wedge n} (v_1,...,v_{2n+1}) = \iota_X \omega \wedge \omega^{\wedge n} (v_1,...,v_{2n+1}) = \frac{1}{n} \omega ^{\wedge (n+1)} (X,v_1,...,v_{2n+1}).
\end{align}
Since $\{X,v_1,...,v_{2n+1}\}$ is a basis of $T_x M$ (due to transversality) and $\omega ^{\wedge (n+1)}$ is a volume form on $M$, we obtain that (\ref{contact}) is non-zero, i.e.,\ the contact condition is satisfied.\

Any hypersurface $\Sigma \subset M$ has a characteristic foliation $L$ which is a rank 1 foliation with $L_x = \text{ker} \left( \omega | _{T_x \Sigma} \right)$, for $x \in \Sigma$.\ If $\Sigma$ is a energy level set of a Hamiltonian $H \colon M \to \mathbb{R}$, then for $x \in \Sigma$ we have that $X_H(x) \in L_x$.\ If $\Sigma$ is of contact type, then $R_{\alpha}(x) \in L_x$, i.e., the Reeb flow of $\alpha$ is a reparametrization of the Hamiltonian flow.\ In the case of $M=T^* Q$, if the contact form on $\Sigma \subset T^*Q$ is induced by the transversality of the natural Liouville vector field $X$ on $T^*Q$, then the contact structure is called the ``standard contact structure'' determined by
$$\xi_{st} = \text{ker} \alpha_{can},\quad \alpha_{can} := \iota_X \omega_{can} | _{\Sigma} = \lambda_{can} | _{\Sigma}.$$
Moreover, in this case the energy hypersurface $\Sigma \subset T^* Q$ is ``fiberwise starshaped'', i.e.,\ for each point $q \in Q$ the intersection $\Sigma \cap T_q^*Q$ bounds a starshaped domain in the linear space $T_q^*Q$, which means that the natural Liouville vector field is transverse to each $\Sigma \cap T_q^*Q$.\

\subsection{Proof of transversality in non-regularized case}

We now consider the Liouville vector field on $T^* \mathbb{R}^3$ given by
\begin{align}\label{liouville_vector_field}
	X = x \frac{\partial}{\partial x} + y \frac{\partial}{\partial y} + z \frac{\partial}{\partial z}.
\end{align}
\begin{Proposition} \label{proposition}
	\textit{For any given $\mu \in [0,\frac{1}{2}]$ assume that $c < H (L_{1/2}) = - \frac{3}{2} \sqrt[3]{\lambda_2}$.\ Then the bounded component $\Sigma_c^b$ of the energy level set, as defined by~(\ref{bounded_component}), is transverse to $X$.}\
\end{Proposition}
As a consequence of Proposition \ref{proposition}, $\iota_X \omega | _{\Sigma_c^b}$ defines a contact form on $\Sigma_c^b$.\ In order to prove Proposition \ref{proposition}, we need some properties of the effective potential (\ref{effective_potential}), which we formulate in three lemmas and discuss in spherical coordinates,
\begin{align*}
	x &= \rho \cos \theta \sin \varphi \\
	y &= \rho \sin \theta \sin \varphi \\
	z &= \rho \cos \varphi
\end{align*}
where $0 \leq \theta \leq 2\pi$, $0 \leq \varphi \leq \pi$.\ Since we consider energy level sets below the first critical value, the radius $\rho$ is always smaller than the distance from $L_{1/2}$ to the origin, which is $1 / \sqrt[3]{\lambda_2}$ and always less than~1.\ Therefore, we assume that the radius $\rho$ is smaller than 1.\ Now the effective potential (\ref{effective_potential}) reads
$$ U (\rho, \theta , \varphi) = - \frac{1}{\rho} - \frac{1}{2} \rho^2 ( \lambda_2 \cos^2 \theta \sin^2 \varphi + \lambda_1 \sin^2 \theta \sin^2 \varphi - \cos^2 \varphi ), $$
which is $\pi$-periodic in the variables $\theta$ and $\varphi$.\
\begin{lemma1}\label{lemma_1_U_min}
	\textit{For fixed radius $\rho \in (0,1)$ the function $U_{\rho} := U(\rho, \cdot, \cdot)$ has its minimum at $(\theta,\varphi) = (0,\frac{\pi}{2})$.}\
\end{lemma1}
\begin{proof}
	The differential is given by
	$$ dU_{\rho} (\theta,\varphi) = \rho^2 (\lambda_2 - \lambda_1) \cos \theta \sin \theta \sin^2 \varphi d \theta + \rho^2 \sin \varphi \cos \varphi (\lambda_2 \cos^2 \theta + \lambda_1 \sin^2 \theta + 1) d \varphi. $$
	Since $\lambda_2 > \lambda_1$, and the term $\lambda_2 \cos^2 \theta + \lambda_1 \sin^2 \theta + 1$ is strictly positive, we find four critical points at $(0,0)$, $(0,\frac{\pi}{2})$, $(\frac{\pi}{2})$ and $(\frac{\pi}{2},\frac{\pi}{2})$.\ The corresponding Hessians are given by
	\begin{align*}
		H_{U_{\rho}} (0,0) &= \begin{pmatrix}
			0 & 0 \\
			0 & -\rho^2(\lambda_2 + 1)
		\end{pmatrix},\quad
		H_{U_{\rho}} (0,\frac{\pi}{2}) = \begin{pmatrix}
			\rho^2 (\lambda_2 - \lambda_1) & 0 \\
			0 & \rho^2(\lambda_2 + 1)
		\end{pmatrix} \\		
		H_{U_{\rho}} (\frac{\pi}{2},0) &= \begin{pmatrix}
			0 & 0 \\
			0 & -\rho^2(\lambda_1 + 1)
		\end{pmatrix},\quad 		
		H_{U_{\rho}} (\frac{\pi}{2},\frac{\pi}{2}) = \begin{pmatrix}
			- \rho^2 (\lambda_2 - \lambda_1) & 0\\
			0 & \rho^2(\lambda_1 + 1)
		\end{pmatrix}.
	\end{align*}
	Therefore, the function $U_{\rho}$ attains its minimum at $(\theta,\varphi) = (0,\frac{\pi}{2})$.\
\end{proof}
We denote by $r := 1 / \sqrt[3]{\lambda_2}$ the distance from $L_{1/2}$ to the origin and introduce
$$ B_r(0) := \{ (x,y,z) \in \mathbb{R}^3 \colon x^2 + y^2 + z^2 \leq r^2 \} $$
the ball of radius $r$ centered at the origin.\
\begin{corollary1}\label{corollary}
	\textit{The bounded part of Hill's region, $\mathscr{K}_c^b$, is contained in $B_r(0)$.}\
\end{corollary1}
\begin{proof}
	Let $(\rho,\theta,\varphi) \in \partial B_r(0)$, i.e.,\ $\rho = r = 1 / \sqrt[3]{\lambda_2}$.\ Then, by Lemma \ref{lemma_1_U_min},
	\begin{align}\label{U>c}
		U(r,\theta,\varphi) \geq U(r,0,\frac{\pi}{2}) = - \frac{1}{r} - \frac{1}{2}r^2 \lambda_2 = - \frac{3}{2}\sqrt[3]{\lambda_2} = H(L_{1/2}) > c.
	\end{align}
	Therefore, $(r,\theta,\varphi)$ does not lie in $\mathscr{K}_c^b$, and hence, $\partial B_r(0) \cap \mathscr{K}_c^b = \emptyset$.\ Since $\mathscr{K}_c^b$ is connected and contains the origin in its closure, $\mathscr{K}_c^b$ is contained in $B_r(0)$.\
\end{proof}
\begin{lemma1}\label{lemma_2_U_pos}
	\textit{For every $(\rho,\theta,\varphi) \in B_r(0)$ with $\rho \in (0,r)$ it holds that $\frac{\partial U}{\partial \rho} (\rho,\theta,\varphi) > 0$.}\
\end{lemma1}
\begin{proof}
	Let $(\rho,\theta,\varphi) \in B_r(0)$ with $\rho \in (0,r)$.\ Since $\lambda_2 > \lambda_1$ we have the following equivalences
	\begin{align}\label{equivalence}
		(\lambda_1 - \lambda_2)\sin^2 \theta \leq 0 \quad \Leftrightarrow \quad \lambda_2(\cos^2 \theta - 1) + \lambda_1 \sin^2 \theta \leq 0 \quad \Leftrightarrow \quad \lambda_2 \cos^2 \theta + \lambda_1 \sin^2 \theta \leq \lambda_2.
	\end{align}
	By using (\ref{equivalence}), we estimate
	\begin{align} \label{partial_U_rho}
		\frac{\partial U}{\partial \rho} = \frac{1}{\rho^2} - \rho \left( \lambda_2 \cos^2 \theta \sin^2 \varphi + \lambda_1 \sin^2 \theta \sin^2 \varphi - \cos^2 \varphi \right) \geq \frac{1}{\rho^2} - \lambda_2 \rho > 0.
	\end{align}
	The last strict inequality holds since the function $f \colon (0,r) \to \mathbb{R}$, $x \mapsto \frac{1}{x^2} - \lambda_2 x$ is strictly positive on its domain.\
\end{proof}
\begin{lemma1} \label{lemma_3_U2}
	\textit{For every $(\rho,\theta,\varphi) \in B_r(0)$ with $\rho > 0$ it holds that $\frac{\partial^2 U}{\partial \rho^2} (\rho,\theta,\varphi) \leq - \sin^2 \varphi$.}\
\end{lemma1}
\begin{proof}
	Let $(\rho,\theta,\varphi) \in B_r(0)$ with $\rho > 0$.\ Since the function $f \colon (0,r] \to \mathbb{R}$, $x \mapsto - \frac{1}{x^3}$ takes the maximal value at $x=r$, and because $\lambda_2 \geq 2$, we estimate
	$$ \frac{\partial^2 U}{\partial \rho^2} = - \frac{2}{\rho^3} + \cos^2 \varphi - \sin^2 \varphi \left( \lambda_2 \cos^2 \theta + \lambda_1 \sin^2 \theta \right) \leq - \frac{2}{r^3} + 1 = -2 \lambda_2 + 1 \leq -3 \leq - \sin^2 \varphi. \qedhere $$
\end{proof}

\begin{proof}[Proof of Proposition \ref{proposition}]
We show that
\begin{align}\label{inequality_to_show_1}
	d H (X) | _{\Sigma_c^b} > 0.
\end{align}
The differential of the Hamiltonian (\ref{hamiltonian_1}) is given by
\begin{align}\label{dH}
	dH =  & p_x dp_x + p_y dp_y + p_z dp_z + p_x dy + y dp_x - p_y dx - x dp_y \\
	& + 2ax dx + 2by dy + z dz + \frac{x}{r^3} dx + \frac{y}{r^3} dy + \frac{z}{r^3} dz. \nonumber
\end{align}
By inserting the Liouville vector field (\ref{liouville_vector_field}) into (\ref{dH}) we obtain
\begin{align}\label{dHX}
	dH(X) = p_x y - p_y x + 2ax^2 + 2by^2 + z^2 + \frac{1}{r}.
\end{align}
Recall that $a = \frac{1}{2}(1-\lambda_2)$ and $b = \frac{1}{2}(1 - \lambda_1)$.\ In spherical coordinates the Liouville vector field (\ref{liouville_vector_field}) becomes
$$X = \rho \frac{\partial}{\partial \rho},$$
and (\ref{dHX}) reads
\begin{align}\label{dHX_2}
	dH(X) = & p_x \rho \sin \theta \sin \varphi - p_y \rho \cos \theta \sin \varphi + (1 + \lambda_2) \rho^2 \cos^2 \theta \sin^2 \varphi \\
	& + (1 - \lambda_1) \rho^2 \sin^2 \theta \sin^2 \varphi + \rho^2 \cos^2 \varphi + \frac{1}{\rho}. \nonumber
\end{align}
In view of $\frac{\partial U}{\partial \rho}$ from (\ref{partial_U_rho}), we write (\ref{dHX_2}) in the form
\begin{align*}
	dH(X) = \rho \sin \theta \sin \varphi (p_x + \rho \sin \theta \sin \varphi) - \rho \cos \theta \sin \varphi (p_y - \rho \cos \theta \sin \varphi) + \rho \frac{\partial U}{\partial \rho},
\end{align*}
which we estimate by using the Cauchy--Schwarz inequality,
\begin{align*}
	dH(X) &\geq \rho \frac{\partial U}{\partial \rho} - \rho \sin \varphi \sqrt{ (p_x + \rho \sin \theta \sin \varphi)^2 + (p_y - \rho \cos \theta \sin \varphi)^2 } \\
	& = \rho \frac{\partial U}{\partial \rho} - \rho \sin \varphi \sqrt{ 2 (H - U) - p_z^2 } \\
	& \geq \rho \frac{\partial U}{\partial \rho} - \rho \sin \varphi \sqrt{2 (H - U)}.
\end{align*}
Therefore, we have
$$ dH(X) | _{\Sigma_c^b} \geq \rho \left( \frac{\partial U}{\partial \rho} - \sin \varphi \sqrt{2(c-U)} \right). $$
Since the right hand side is independent of the momentum coordinates, to prove (\ref{inequality_to_show_1}) it is suffices to show that
\begin{align}\label{inequality_to_show_2}
	\left( \frac{\partial U}{\partial \rho} - \sin \varphi \sqrt{2(c-U)} \right) \bigg| _{\mathscr{K}_c^b} > 0.
\end{align}
Let $(\rho,\theta,\varphi) \in \mathscr{K}_c^b$.\ In particular, $ U(\rho,\theta,\varphi) \leq c$.\ By Corollary \ref{corollary}, we have $\rho < r$, and by (\ref{U>c}) it holds that $U(r,\theta,\varphi) > c$.\ Therefore, it exists $\tau \in [0,r-\rho)$ such that
$$ U(\rho + \tau, \theta, \varphi) = c. $$
By using Lemma \ref{lemma_2_U_pos} and Lemma \ref{lemma_3_U2} we obtain
\begin{align*}
	\left( \frac{\partial U}{\partial \rho} (\rho,\theta,\varphi) \right)^2 &= \left( \frac{\partial U}{\partial \rho} (\rho + \tau,\theta,\varphi) \right)^2 - \int_{0}^{\tau} \frac{d}{dt} \left( \frac{\partial U}{\partial \rho} (\rho + t,\theta,\varphi) \right)^2 dt \\
	&> -2 \int_{0}^{\tau} \frac{\partial U}{\partial \rho} (\rho + t,\theta,\varphi) \frac{\partial^2 U}{\partial \rho^2} (\rho + t,\theta,\varphi) dt \\
	&\geq 2 \sin^2 \varphi \int_{0}^{\tau} \frac{\partial U}{\partial \rho} (\rho + t,\theta,\varphi) dt \\
	&= 2 \sin^2 \varphi \left( U(\rho + \tau,\theta,\varphi) - U(\rho,\theta,\varphi) \right) \\
	&= 2 \sin^2 \varphi \left( c - U(\rho,\theta,\varphi) \right).
\end{align*}
Therefore, by using Lemma \ref{lemma_2_U_pos} once more, we imply
$$ \frac{\partial U}{\partial \rho} (\rho,\theta,\varphi) > \sin \varphi \sqrt{ 2 \left( c - U(\rho,\theta,\varphi) \right)}, $$
which shows (\ref{inequality_to_show_2}) and thereby the proposition.\
\end{proof}

\subsection{Moser-regularized energy level set and proof of transversality near the origin}

The Hamiltonian (\ref{hamiltonian_1}) has a singularity at the origin corresponding to collisions, thus the bounded component $\Sigma_c^b$ of the energy level set is non-compact.\ Moser \cite{moser} observed that the regularized Kepler problem coincides with the geodesic flow on the sphere endowed with its standard metric by interchanging the roles of position and momenta.\ To remove the singularity in our problem, we use the same concept as introduced by Moser.\

We abbreviate by $\mathbf{X} = (x,y,z)$ and $\mathbf{P}=(p_x,p_y,p_z)$ the corresponding position and momentum coordinates.\ We use a new time parameter $s$ and define for an energy value $c < H(L_{1/2}) = - \frac{3}{2}\sqrt[3]{\lambda_2}$ a new Hamiltonian by
$$ s = \int \frac{dt}{|\mathbf{X}|},\quad K_c(\mathbf{X},\mathbf{P}) := |\mathbf{X}| \left( H(\mathbf{X},\mathbf{P}) - c \right), $$
Notice that the flow of $H$ at energy level $c$ corresponds to the flow of $K_c$ at energy level $0$.\ Now we interchange the roles of position and momenta by the symplectic transformation mapping $(\mathbf{X},\mathbf{P})$ to $(-\mathbf{P},\mathbf{X})$.\ For simplicity of notation, we replace the new coordinates $\mathbf{X'}=-\mathbf{P}$ and $\mathbf{P'}=\mathbf{X}$ by $\mathbf{X}$ and $\mathbf{P}$.\ Then, the new transformed Hamiltonian $\widetilde{K}_c(\mathbf{X},\mathbf{P}) = K_c(-\mathbf{P},\mathbf{X})$ is explicitly given by
\begin{align}\label{hamiltonian_Kc}
	\widetilde{K}_c (\mathbf{X},\mathbf{P}) &= \frac{1}{2}|\mathbf{X}|^2|\mathbf{P}| + |\mathbf{P}|(p_x y - p_y x) - 1 + |\mathbf{P}|(a p_x^2 + b p_y^2 + \frac{1}{2}p_z^2) - |\mathbf{P}|c \\
	&= \frac{1}{2} \left( |\mathbf{X}|^2 + 1 \right)|\mathbf{P}| + (p_x y - p_y x)|\mathbf{P}| -1 + (a p_x^2 + b p_y^2 + \frac{1}{2}p_z^2)|\mathbf{P}| - (c+\frac{1}{2})|\mathbf{P}|. \nonumber
\end{align}
The next step is to use the stereographic projection which induces a symplectic transformation between  $T^*\mathbb{R}^3$ and $T^*S^3$ that extends $\widetilde{K}_c$ to a Hamiltonian on $T^*S^3$.\ Let $\xi = (\xi_0,\xi_1,\xi_2,\xi_3) \in \mathbb{R}^4$ with norm 1.\ We write a tangent vector $\eta \in T_{\xi}S^3$ as $\eta = (\eta_0,\eta_1,\eta_2,\eta_3)$, with inner product $(\xi,\eta)=0$.\ We identify $TS^3$ with $T^*S^3 \subset T^*\mathbb{R}^4$ by using the standard metric on $S^3$.\ Then, the symplectic transformation is given by
\begin{gather}\label{stereographic_symplectic}
	x = \frac{\xi_1}{1-\xi_0},\quad y = \frac{\xi_2}{1-\xi_0},\quad z = \frac{\xi_3}{1-\xi_0}, \\
	p_x = \eta_1(1-\xi_0) + \xi_1 \eta_0,\quad p_y = \eta_2(1-\xi_0) + \xi_2 \eta_0,\quad p_z = \eta_3(1-\xi_0) + \xi_3 \eta_0.\nonumber
\end{gather}
Notice that here $(x,y,z)$ represents the momentum and $(p_x,p_y,p_z)$ the position compared to the original picture before switching their roles.\ After this transformation, going to the North pole (where the momentum becomes infinite) corresponds to collision in the original picture (where the position becomes zero).\ Dynamically, at collision (going through the North pole) it bounces back.\ Therefore, Moser regularization is characterized by adding the fiber over the North pole.\ Moreover, the inverse transformation is given by
\begin{gather*}
	\xi_0 = \frac{|\mathbf{X}|^2 - 1}{|\mathbf{X}|^2 + 1},\quad \xi_1 = \frac{2x}{|\mathbf{X}|^2 + 1},\quad \xi_2 = \frac{2y}{|\mathbf{X}|^2 + 1},\quad \xi_3 = \frac{2z}{|\mathbf{X}|^2 + 1}, \\
	\eta_0 = \langle \mathbf{X},\mathbf{P} \rangle,\quad \eta_1 = \frac{|\mathbf{X}|^2+1}{2}p_x - \langle \mathbf{X},\mathbf{P} \rangle x,\quad \eta_2 = \frac{|\mathbf{X}|^2+1}{2}p_y - \langle \mathbf{X},\mathbf{P} \rangle y,\quad \eta_3 = \frac{|\mathbf{X}|^2+1}{2}p_z - \langle \mathbf{X},\mathbf{P} \rangle z,
\end{gather*}
and, in addition, we have the relation
\begin{align}\label{relation}
	|\eta| = \frac{1}{2}(|\mathbf{X}|^2 + 1)|\mathbf{P}| = \frac{|\mathbf{P}|}{1 - \xi_0}.
\end{align}
By inserting (\ref{stereographic_symplectic}) and (\ref{relation}) into (\ref{hamiltonian_Kc}), the transformed Hamiltonian on $T^*S^3$, which we denote by the same letter, is given by
\begin{align}\label{ham_K_c_xi_eta}
	\widetilde{K}_c (\xi,\eta) = |\eta| f(\xi,\eta) - 1,
\end{align}
where
\begin{gather*}
	f(\xi,\eta) := 1 + (\eta_1 \xi_2 - \eta_2 \xi_1)(1-\xi_0) + (a g_1^2 + b g_2^2 + \frac{1}{2} g_3^2)(1 - \xi_0) - (c + \frac{1}{2})(1 - \xi_0), \\
	g_k := g_k(\xi,\eta) := \eta_k(1 - \xi_0) + \xi_k \eta_0,\quad k=1,2,3.
\end{gather*}
By shifting and squaring the Hamiltonian (\ref{ham_K_c_xi_eta}) we obtain the new smooth Hamiltonian $Q(\xi,\eta)$ on a subset of $T^*S^3$,
\begin{align}\label{Hamiltonian_Q}
	Q(\xi,\eta) = \frac{1}{2}|\eta|^2 f(\xi,\eta)^2.
\end{align}
The level set $H^{-1}(c) = K_c^{-1}(0)$ is compactified to the level set $Q^{-1}(\frac{1}{2})$.\ Since $Q$ is smooth near this level set, we consider $Q^{-1}(\frac{1}{2})$ as the regularized problem.\ Since the only problem in compactness of $\Sigma_c^b$ comes from collisions with the origin, we consider points near the origin, i.e.,\ in view of (\ref{relation}), points $(\xi,\eta)$ satisfying
\begin{align}\label{points_near_origin}
	|\mathbf{P}|=|\eta|(1-\xi_0) < \varepsilon.
\end{align}
\begin{Proposition}\label{prop}
	\textit{For $\varepsilon > 0 $ small enough, the natural Liouville vector field on $T^*S^3$ given by}
	\begin{align} \label{liouville_v_f}
		X = \sum_{i = 0}^{3} \eta_i \frac{\partial}{\partial \eta_i},
	\end{align}
	\textit{is transverse to $Q^{-1}(\frac{1}{2})$ over points $(\xi,\eta)$ satisfying (\ref{points_near_origin}).}
\end{Proposition}
Notice that the Liouville vector field (\ref{liouville_vector_field}) on $T^*\mathbb{R}^3$ that we used for transversality in the unregularized case is mapped, via the composition of the symplectic transformation (\ref{stereographic_symplectic}) with the symplectic switch map, to the natural Liouville vector field (\ref{liouville_v_f}) on $T^*S^3$.\
\begin{proof}[Proof of Proposition \ref{prop}]
	We show that for $\varepsilon > 0$ small enough it holds that
	\begin{align}\label{to_show}
		dQ(X) | _{Q^{-1}(\frac{1}{2})} > 0.
	\end{align}
	The computation of $dQ(X)$, in view of (\ref{Hamiltonian_Q}) and (\ref{liouville_v_f}), yields
	\begin{align*}
		dQ(X) &= |\eta|^2 f(\xi,\eta)^2 + |\eta|^2 f(\xi,\eta) \sum_{i=0}^{3} \frac{\partial f}{\partial \eta_i} (\xi,\eta) \eta_i \\
		&= 2 Q + |\eta|^2 f(\xi,\eta) (1-\xi_0) ( \eta_1 \xi_2 - \eta_2 \xi_1 + 2a g_1^2 + 2b g_2^2 + g_3^2 ).
	\end{align*}
	In order to prove (\ref{to_show}), we first show that we can choose $\varepsilon > 0$ so small such that
	\begin{align}\label{bound_lower_f}
		|f(\xi,\eta)| \geq \frac{1}{2}.
	\end{align}
	Since the energy value $c < H(L_{1/2}) = -\frac{3}{2}\sqrt[3]{\lambda_2}$ is negative, and in fact less then $-\frac{3}{2}$, the quantity $c + \frac{1}{2}$ is negative as well.\ Notice from Figure \ref{figure_2_quantities} that $a<0$, $|a|\leq 1$ and $b>0$.\ Therefore, $b g_2^2 + \frac{1}{2} g_3^2 - (c + \frac{1}{2})$ is positive.\ By using these, we estimate
	\begin{align*}
		|f(\xi,\eta)| &= \left| 1 + (\eta_1 \xi_2 - \eta_2 \xi_1)(1-\xi_0) + (a g_1^2 + b g_2^2 + \frac{1}{2} g_3^2)(1 - \xi_0) - (c + \frac{1}{2})(1 - \xi_0) \right| \\
		&= \left| 1 + ( b g_2^2 + \frac{1}{2} g_3^2 - (c + \frac{1}{2}) )(1-\xi_0) + (\eta_1 \xi_2 - \eta_2 \xi_1)(1-\xi_0) + a g_1^2(1 - \xi_0) \right| \\
		&\geq 1 - |\eta_1 \xi_2 - \eta_2 \xi_1|(1-\xi_0) - |a|g_1^2(1-\xi_0) \\
		&\geq 1 - |\eta_1 \xi_2 - \eta_2 \xi_1|(1-\xi_0) - g_1^2(1-\xi_0).
	\end{align*}
	Furthermore, $|\eta_1 \xi_2 - \eta_2 \xi_1| \leq |\eta||\xi|$, and because $|\xi|=1$, we have in view of (\ref{points_near_origin}),
	\begin{align}\label{inequality_epsilon}
		|\eta_1 \xi_2 - \eta_2 \xi_1|(1-\xi_0) \leq |\eta|(1-\xi_0) < \varepsilon.
	\end{align}
	This implies,
	$$ |f(\xi,\eta)| \geq 1 - \varepsilon - g_1^2(1-\xi_0). $$
	If $\varepsilon$ approaches 0, then $\xi_0 \to 1$, which means that we can choose $\varepsilon$ so small such that (\ref{bound_lower_f}) holds.\ By using the level set condition $Q^{-1}(\frac{1}{2})$ together with the lower bound (\ref{bound_lower_f}) for $|f(\xi,\eta)|$, we find
	$$ \frac{1}{2} = Q(\xi,\eta) = \frac{1}{2} |\eta|^2 f(\xi,\eta)^2 \geq \frac{1}{2} |\eta|^2 \frac{1}{2}, $$
	which gives an upper bound for $|\eta|$, i.e.,\
	\begin{align}\label{bound_eta}
		|\eta| \leq 2.
	\end{align}
	We may write
	$$ dQ(X) \geq 2 Q - |\eta|^2 \left| f(\xi,\eta) \right| \left| (1-\xi_0) \left( \eta_1 \xi_2 - \eta_2 \xi_1 + 2a g_1^2 + 2b g_2^2 + g_3^2 \right) \right|. $$
	Notice that by (\ref{bound_eta}) we obtain
	$$|\eta||\eta||f(\xi,\eta)| \leq 2 \sqrt{ 2Q(\xi,\eta) } = 2 \sqrt{2 \frac{1}{2}} = 2, $$
	which implies, together with (\ref{inequality_epsilon}),
	\begin{align*}
		dQ(X) &\geq 1 -2 \left( \left| (1-\xi_0) (\eta_1 \xi_2 - \eta_2 \xi_1) \right| + \left| (1-\xi_0) (2a g_1^2 + 2b g_2^2 + g_3^2 ) \right| \right) \\
		&\geq 1 -2 \varepsilon \left(1 + |2a g_1^2 + 2b g_2^2 + g_3^2| \right).
	\end{align*}
	Since the latter term can be bounded by some constant $A$ on a compact set away from the origin, we obtain
	$$dQ(X) \geq 1 - 2\varepsilon (1 + A). $$
	Now we choose $\varepsilon$ sufficiently small such that $dQ(X)>0$, which proves (\ref{to_show}).\
\end{proof}


We have seen that for $c < H(L_{1/2})$ the bounded component $\Sigma_c^b$ of the energy level set can be Moser-regularized to form a compact 5-dimensional manifold $\widetilde{\Sigma}_c^b \subset T^*S^3$ which is diffeomorphic to $S^*S^3$.\ Since the Liouville vector field (\ref{liouville_vector_field}) on $T^*\mathbb{R}^3$ and the natural one (\ref{liouville_v_f}) on $T^*S^3$ coincide after Moser regularization, we obtain a Liouville vector field that is defined near the whole regularized level set, and in fact, it is the natural one.\ By the transversality results from Proposition \ref{proposition} and Proposition \ref{prop} we obtain that the natural Liouville vector field on $T^*S^3$ is transverse to $\widetilde{\Sigma}_c^b$, which means that $\widetilde{\Sigma}_c^b$ is fiberwise starshaped, and moreover, $\widetilde{\Sigma}_c^b \cong (S^*S^3,\xi_{st})$.\

For the planar problem, one can of course perform the same computation to obtain the same result.\ But since the planar problem corresponds to the restriction of the spatial system to the fixed point set of the symplectic symmetry $\sigma$ from (\ref{sigma}), the transversality result in the planar case follows immediately.\ This consequence is based on a general construction from \cite{aydin_sympl_splitting}.\ Namely, if a energy level set $\Sigma$ is of contact type and the entire system has a symplectic symmetry $\sigma$, such as~(\ref{sigma}), then the restriction of the contact form on $\Sigma$ to $\Sigma | _{\text{Fix}(\sigma | _{\Sigma})}$ is a contact form on $\Sigma | _{\text{Fix}(\sigma | _{\Sigma})}$.\ Therefore, we have the same result in the planar problem, in which we denote the Moser-regularized compact 3-dimensional manifold by $\widetilde{\Sigma}_c^b|_{\text{Fix}(\sigma)} \cong S^* S^2 \subset T^*S^2$.\ This completes the proof of Theorem \ref{theorem}.\\
\\
\textbf{Acknowledgement.}\ The author acknowledges support by the Deutsche Forschungsgemeinschaft (DFG, German Research Foundation), Project-ID 541062288, the Transregional Colloborative Research Center CRC/TRR 191 (281071066), and Germany’s Excellence Strategy EXC- 2181/1 - 390900948 (the Heidelberg STRUCTURES Excellence Cluster).\ Moreover, the author wish to thank the anonymous referee for bringing recent results on dynamics in low-dimensional contact topology to his attention.\


\addcontentsline{toc}{section}{References}
\begin{thebibliography}{1}
			
			
			
			\bibitem{albers_frauenfelder_koert_paternain}
			Albers P., Frauenfelder U., van Koert O., Paternain G.:
			\textit{Contact geometry of the restricted three-body problem}.
			Comm. Pure Appl. Math. \textbf{65}(2), 229--263 (2012) \url{https://doi.org/10.1002/cpa.21380}
			
			\bibitem{almeida}
			Almeida A.J.C., Peixinho N., Correia A.C.M.:
			\textit{Neptune Trojans and Plutinos: colors, sizes, dynamics, and their possible collisions}.
			Astron. Astrophys. \textbf{508}, 1021--1030 (2009) \url{https://doi.org/10.1051/0004-6361/200911943 }
			
			\bibitem{alvarez_vidal}
			Alvarez-Ramirez M., Vidal C.:
			\textit{Dynamical aspects of an equilateral restricted four-body problem}.
			Math. Probl. Eng. \textbf{2009}(181360) (2009) \url{https://doi.org/10.1155/2009/181360}
			
			
			\bibitem{aydin_babylonian}
			Aydin C.:
			\textit{From Babylonian lunar observations to Floquet multipliers and Conley--Zehnder Indices}.
			J. Math. Phys. \textbf{64}(8) (2023) \url{https://doi.org/10.1063/5.0156959}
			
			\bibitem{aydin_sympl_splitting}
			Aydin C.:
			\textit{Symplectic splitting of Hamiltonian structures and reduced monodromy matrices}.
			Linear Algebra Appl. \textbf{674}, 426--441 (2023) \url{https://doi.org/10.1016/j.laa.2023.06.008}
			
			\bibitem{aydin_cz}
			Aydin C.:
			\textit{The Conley--Zehnder indices of the spatial Hill three-body problem}.
			Celest. Mech. Dyn. Astron. \textbf{135}(32) (2023) \url{https://doi.org/10.1007/s10569-023-10134-7}
			
			\bibitem{aydin_sym}
			Aydin C.:
			\textit{The linear symmetries of Hill’s lunar problem}.
			Arch. Math. (Basel) \textbf{120}(3), 321–330 (2023) \url{https://doi.org/10.1007/s00013-022-01822-1}
			
			\bibitem{baltagiannis_2011}
			Baltagiannis A.N., Papadakis K.E.:
			\textit{Equilibrium points and their stability in the restricted four-body problem}.
			Internat. J. Bifur. Chaos \textbf{21}(8), 2179--2193 (2011) \url{https://doi.org/10.1142/S0218127411029707}
			
			\bibitem{baltagiannis_2013}
			Baltagiannis A.N., Papadakis K.E.:
			\textit{Periodic solutions in the Sun–Jupiter–Trojan Asteroid–Spacecraft system}.
			Planet. Space Sci. \textbf{75}, 148--157 (2013) \url{https://doi.org/10.1016/j.pss.2012.11.006}
			
			
			
			
			\bibitem{birkhoff}
			Birkhoff G.D.:
			\textit{The restricted problem of three bodies}.
			Rend. Circ. Mat. Palermo \textbf{39}, 265--334 (1915) \url{https://doi.org/10.1007/BF03015982}
			
			\bibitem{burgos_celletti}
			Burgos-García J., Celletti A., Gales C., Gidea M., Lam W.T.:
			\textit{Hill four-body problem with oblate bodies: an application to the Sun-Jupiter-Hektor-Skamandrios system}.
			J. Nonlinear Sci. \textbf{30}(6), 2925--2970 (2020) \url{https://doi.org/10.1007/s00332-020-09640-x}
			
			\bibitem{burgos_gidea}
			Burgos-García J., Gidea M.:
			\textit{Hill’s approximation in a restricted four-body problem}.
			Celest. Mech. Dyn. Astron. \textbf{122}, 117-141 (2015) \url{https://doi.org/10.1007/s10569-015-9612-9}
						
			\bibitem{burgos_delgado}
			Burgos-García J., Delgado J.:
			\textit{Periodic orbits in the restricted four-body problem with two equal masses}.
			Astrophys. Space Sci. \textbf{345}, 247--263 (2012) \url{https://doi.org/10.1007/s10509-012-1118-2}
			
			
			
			
			
			
			\bibitem{cho_jung_kim}
			Cho W., Jung H., Kim G.:
			\textit{The contact geometry of the spatial circular restricted 3‑body problem}.
			Abh. Math. Semin. Univ. Hambg. \textbf{90}, 161--181 (2020) \url{https://doi.org/10.1007/s12188-020-00222-y}
			
			\bibitem{colin_dehornoy_hryniewicz_rechtman}
			Colin V., Dehornoy P., Hryniewicz U., Rechtman A.:
			\textit{Generic properties of 3-dimensional Reeb flows: Birkhoff sections and entropy}.
			Comment. Math. Helv. \textbf{99}(3), 557--611 (2024) \url{https://doi.org/10.4171/CMH/573}
			
			\bibitem{colin_dehornoy_rechtman}
			Colin V., Dehornoy P., Rechtman A.:
			\textit{On the existence of supporting broken book decompositions for contact forms in dimension 3}.
			Invent. math. \textbf{231}, 1489--1539 (2023) \url{https://doi.org/10.1007/s00222-022-01160-7}
			
			
			
			\bibitem{connors}
			Connors M. Stacey G., Brasser R., Wiegert P.:
			\textit{A survey of orbits of co-orbitals of Mars}.
			Planet. Space Sci. \textbf{53}, 617--624 (2005) \url{https://doi.org/10.1016/j.pss.2004.12.004}
			
			\bibitem{cristofaro_hryn_hutchings_liu}
			Cristofaro-Gardiner D., Hryniewicz U., Hutchings M., Liu H.:
			\textit{Contact three-manifolds with exactly two simple Reeb orbits}.
			Geom. Topol. \textbf{27}(9), 3801--3831 (2023) \url{https://doi.org/10.2140/gt.2023.27.3801}
			
			\bibitem{cronin}
			Cronin J., Richards P.B., Russell L.H.:
			\textit{Some periodic solutions of a four-body problem}.
			Icarus \textbf{3}, 423--428 (1964) \url{https://doi.org/10.1016/0019-1035(64)90003-X}
			
			
			\bibitem{floer}
			Floer A.:
			\textit{Morse theory for Lagrangian intersection}.
			 J. Differential Geom. \textbf{28}(3), 513--547 (1988) \url{https://doi.org/10.4310/jdg/1214442477}
			
			\bibitem{frauenfelder_koert}
			Frauenfelder U., van Koert O.:
			\textit{The Restricted Three-Body Problem and Holomorphic Curves}.
			Pathw. Math., Birkhäuser (2018)
			
			
			\bibitem{gabern_jorba}
			Gabern F., Jorba A.:
			\textit{A restricted four-body model for the dynamics near the Lagrangian points of the Sun-Jupiter system}.
			Discrete Contin. Dyn. Syst. Series B \textbf{1}(2), 143--182(2001) \url{https://doi.org/10.3934/dcdsb.2001.1.143}
			
			
			
			\bibitem{geiges}
			Geiges, H.:
			\textit{An Introduction to Contact Topology}.
			Cambridge studies in advanced mathematics \textbf{109}, Cambridge University Press (2008)
			
			
			
			
			
			
			
			
			
			
			\bibitem{hill}
			Hill G.\textit{}W.:
			\textit{Researches in the Lunar Theory}.
			Am. J. Math. \textbf{1}(3), 245--260 (1878) \url{https://doi.org/10.2307/2369430}
			
			
			\bibitem{hofer}
			Hofer H.:
			\textit{Pseudoholomorphic curves in symplectizations with applications to the Weinstein conjecture in dimension three}.
			Invent. Math. \textbf{114}, 515--563 (1993) \url{https://doi.org/10.1007/BF01232679}
			
			\bibitem{hofer_w_z}
			Hofer H., Wysocki K., Zehnder E.:
			\textit{The Dynamics on Three-Dimensional Strictly Convex Energy Surfaces}.
			Ann. of Math. (2) \textbf{148}(1), 197–289 (1998) \url{https://doi.org/10.2307/120994}
			
			
			
			
			
			
			
			
			\bibitem{irie}
			Irie K.:
			\textit{Strong closing property of contact forms and action selecting functors}.
			J. Fixed Point Theory Appl. \textbf{26}(15) (2024) \url{https://doi.org/10.1007/s11784-024-01102-1}
			
			\bibitem{kumari}
			Kumari R., Kushvah B.S.:
			\textit{Stability regions of equilibrium points in restricted four-body problem with oblateness effects}.
			Astrophys. Space Sci. \textbf{349}, 693--704 (2014) \url{https://doi.org/10.1007/s10509-013-1689-6}
			
			\bibitem{lagrange}
			Lagrange J.L.:
			\textit{Essai sur le problème des trois corps}.
			Oeuvres \textbf{6}, 229--331 (1772)
			
			\bibitem{lamb_roberts}
			Lamb J.S.W., Roberts J.A.G.:
			\textit{Time-reversal symmetry in dynamical systems: A survey}.
			Physica D \textbf{112}, 1--39 (1998) \url{https://doi.org/10.1016/S0167-2789(97)00199-1}
			
			
			
			
			
			
			\bibitem{michalodimitrakis}
			Michalodimitrakis M.:
			\textit{The circular restricted four-body problem}.
			Astrophys. Space Sci. \textbf{75}, 289--305 (1981) \url{https://doi.org/10.1007/BF00648643}
			
			
			\bibitem{moreno_aydin_koert_frauenfelder_koh}
			Moreno A., Aydin C., van Koert O., Frauenfelder U., Koh D.:
			\textit{Bifurcation graphs for the CR3BP via symplectic methods}.
			J. Astronaut. Sci. \textbf{71}(51) (2024) \url{https://doi.org/10.1007/s40295-024-00462-7}
			
			\bibitem{moreno_koert}
			Moreno A., van Koert O.:
			\textit{Global hypersurfaces of section in the spatial restricted three-body problem}.
			Nonlinearity \textbf{35}(6), 2920--2970 (2022) \url{https://doi.org/10.1088/1361-6544/ac692b}
			
			\bibitem{moser}
			Moser J.:
			\textit{Regularization of kepler's problem and the averaging method on a manifold}.
			Comm. Pure Appl. Math. \textbf{23}(4), 609--636 (1970) \url{https://doi.org/10.1002/cpa.3160230406}
			
			\bibitem{moulton}
			Moulton F.R.:
			\textit{On a class of particular solutions of the problem of four bodies}.
			Trans. Am. Math. Soc. 1(1), 17–29 (1900)
			
			\bibitem{murray}
			Murray C.D., Cooper N.J., Evans M.W., Beurle K.:
			\textit{S/2004 S 5: A new co-orbital companion for Dione}.
			Icarus \textbf{179}(1), 222--234 (2005) \url{https://doi.org/10.1016/j.icarus.2005.06.009}
			
			\bibitem{nicholson}
			Nicholson S.B.:
			\textit{The Trojan Asteroids}.
			Publ. Astr. Soc. Pacific Leaflets \textbf{8}(381), 239--246 (1961)
			
			\bibitem{olkin}
			Olkin C., Vincent M., Adam C. et al.:
			\textit{Mission Design and Concept of Operations for the Lucy Mission}.
			Space. Sci. Rev. \textbf{220}(47) (2024) \url{https://doi.org/10.1007/s11214-024-01082-1} 
			
			\bibitem{poincare}
			Poincaré H.:
			\textit{Les méthodes nouvelles de la mécanique céleste}.
			I–III, Gauthiers-Villars, Paris (1892)
			
			\bibitem{poincare_2}
			Poincaré H.:
			\textit{Sur un théorème de géométrie}.
			Rend. Circ. Mat. Palermo \textbf{33}, 375--407 (1912) \url{https://doi.org/10.1007/BF03015314}
			
			
			
			\bibitem{santana}
			Santana-Ros T., Micheli M., Faggioli L. et al.:
			\textit{Orbital stability analysis and photometric characterization of the second Earth Trojan asteroid 2020 $\text{XL}_5$}.
			Nat. Commun. \textbf{13}(447) (2022) \url{https://doi.org/10.1038/s41467-022-27988-4}
			
			\bibitem{scheeres}
			Scheeres D.J.:
			\textit{The Restricted Hill Four-Body Problem with Applications to the Earth–Moon–Sun System}.
			Celest. Mech. Dyn. Astron. \textbf{70}, 75-98 (1998) \url{https://doi.org/10.1023/A:1026498608950}
			
			
			\bibitem{schwarz_dvorak}
			Schwarz R., Dvorak R., Süli A., Erdi B.:
			\textit{Survey of the stability region of hypothetical habitable Trojan planets}.
			Astron. Astrophys. \textbf{474}(3), 1023--1029 (2007) \url{https://doi.org/10.1051/0004-6361:20077994}
			
			
			
			\bibitem{stacey_connors}
			Stacey, R.G., Connors M.:
			\textit{A centenary survey of orbits of co-orbitals of Jupiter}.
			Planet. Space Sci. \textbf{56}, 358--367 (2008) \url{https://doi.org/10.1016/j.pss.2007.11.002}
			
			\bibitem{steves_roy}
			Steves B.A., Roy A.E.:
			\textit{Some special restricted four-body problems—I. Modelling the Caledonian problem}.
			Planet. Space Sci. \textbf{46}, 1465--1474 (1998) \url{https://doi.org/10.1016/S0032-0633(98)00077-4}
			
			\bibitem{taubes}
			Taubes C.H.:
			\textit{The Seiberg–Witten equations and the Weinstein conjecture}.
			Geom. Topol. \textbf{11}(4), 2117-2202 (2007) \url{https://doi.org/10.2140/gt.2007.11.2117}
			
			
			
			\bibitem{wintner}
			Wintner A.:
			\textit{On Hill's Periodic Lunar Orbit}.
			Amer. J. Math. \textbf{60}(4), 937--948 (1938) \url{https://doi.org/10.2307/2371269}
\end{thebibliography}
\end{document}